\documentclass{amsart}
\usepackage{amssymb,euscript,amsmath, mathrsfs, graphicx, subfigure}

\def\RR{{\mathbb R}}
\def\ZZ{{\mathbb Z}}

\def\Eff{\operatorname{Eff}}
\def\Prin{\operatorname{Prin}}

\newcommand{\Div}{\operatorname{Div}\nolimits}

\DeclareMathOperator{\divisor}{div}
\DeclareMathOperator{\ord}{ord}
\DeclareMathOperator{\Jac}{Jac}
\DeclareMathOperator{\Pic}{Pic}

\numberwithin{equation}{section}
\newtheorem{thm}[equation]{Theorem}
\newtheorem{prop}[equation]{Proposition}
\newtheorem{lem}[equation]{Lemma}
\newtheorem{cor}[equation]{Corollary}

\theoremstyle{definition}
\newtheorem{defn}[equation]{Definition}
\newtheorem*{defn*}{Definition}

\newtheorem{ex}[equation]{Example}
\theoremstyle{remark}

\newtheorem{rem}[equation]{Remark}

\begin{document}

\title[Note on Brill-Noether theory for metric graphs]{A note on Brill-Noether theory and rank determining sets for metric graphs}
\author{Chang Mou Lim}
\author{Sam Payne}
\address{Department of Mathematics, Yale University, 10 Hillhouse Ave, New Haven CT, 06511 \ \ \tt changmou.lim@yale.edu \ \ sam.payne@yale.edu \ \ natasha.potashnik@yale.edu}
\author{Natasha Potashnik}

\begin{abstract}
We produce open subsets of the moduli space of metric graphs without separating edges where the dimensions of Brill-Noether loci are larger than the corresponding Brill-Noether numbers.  These graphs also have minimal rank-determining sets that are larger than expected, giving couterexamples to a conjecture of Luo.  Furthermore, limits of these graphs have Brill-Noether loci of the expected dimension, so dimensions of Brill-Noether loci of metric graphs do not vary upper semicontinuously in families.

Motivated by these examples, we study a notion of rank for the Brill-Noether locus of a metric graph, closely analogous to the Baker-Norine definition of the rank of a divisor.  We show that ranks of Brill-Noether loci vary upper semicontinuously in families of metric graphs and are related to dimensions of Brill-Noether loci of algebraic curves by a specialization inequality.
\end{abstract}

\maketitle

\section{Introduction}

Baker and Norine defined linear equivalence and ranks of divisors on graphs and showed that they have many properties similar to those of linear equivalence and dimensions of complete linear series on algebraic curves.  In particular, they satisfy analogues of the Riemann-Roch and Abel-Jacobi Theorems \cite{BakerNorine07}.  These definitions and results extend to metric graphs \cite{GathmannKerber08, MikhalkinZharkov08} and the rank of a divisor on an ordinary graph is equal to its rank on the associated metric graph in which all edges have length 1 \cite{HladkyKralNorine10}.  We follow the usual convention from tropical geometry by referring to the first Betti number of a graph as its genus.

Let $\Gamma$ be a metric graph of genus $g$.  The group $\Pic_0(\Gamma)$ is a real torus of dimension $g$ parametrizing equivalence classes of divisors of degree zero on $\Gamma$, and $\Pic_d(\Gamma)$ is the torsor over $\Pic_0(\Gamma)$ that parametrizes equivalence classes of divisors of degree $d$.   The \textbf{Brill-Noether locus} 
\[
W^r_d(\Gamma) \subset \Pic_d(\Gamma)
\]
is the subset parametrizing divisor classes of rank at least $r$.  This is a closed polyhedral subset of $\Pic_d(\Gamma)$; see Section~\ref{sec:BNloci}.  We consider the dimension of $W^r_d$ as a function on the moduli space of metric graphs of genus $g$, as studied by Culler and Vogtmann \cite{CullerVogtmann86}.  This corresponds to the open subset of the moduli space of tropical curves studied in \cite{BrannettiMeloViviani11} where all vertices are marked with zero.  See also \cite{GathmannKerberMarkwig09, Kozlov09, Caporaso11, Chan11} for other approaches to moduli of tropical curves and metric graphs.

\begin{thm}  \label{thm:notsemicont}
The function taking a metric graph $\Gamma$ to $\dim W^r_d(\Gamma)$ is not upper semicontinuous on the moduli space of metric graphs of genus 4.
\end{thm}

\noindent  This differs from the analogous situation in algebraic geometry.  There is a universal Brill-Noether locus $W^r_d$ in the universal Jacobian, proper over the moduli space $M_g$ of smooth projective algebraic curves of genus $g$ \cite[Chapter~4]{ACGH} so, by general results on fiber dimension \cite[Theorem~13.1.3]{EGA4.3}, the function $\dim W^r_d$ is upper semicontinuous in the Zariski topology.  In particular, there is a dense open subset of $M_g$ where $\dim W^r_d$ achieves its minimum.  In the interesting cases, where this minimum is nonnegative but less than $g$, this minimum is the Brill-Noether number
\[
\rho(g,r,d) = g-(r+1)(g-d+r).
\]

There are open sets in the moduli space of metric graphs of genus $g$ where $\dim W^r_d$ is strictly larger than $\rho(g,r,d)$ for trivial reasons;  for instance, the locus of graphs constructed by taking a trivalent tree with $g$ leaves and attaching a loop to each leaf is open, and all such graphs are hyperelliptic.  However, contractions of separating edges induce isomorphisms on Picard groups and Brill-Noether loci \cite{MikhalkinZharkov08, CaporasoViviani10, BakerFaber11}, and the map taking a metric graph to the graph obtained by contracting all of its separating edges is a strong deformation retract onto the moduli space of metric graphs without separating edges.  This retraction collapses the open sets of hyperelliptic graphs described above onto a set of high codimension.  The relevant question, therefore, is whether the Brill-Noether number is equal to the dimension of the Brill-Noether locus on a dense subset of the moduli space of metric graphs without separating edges.  Still, the answer is negative.

\begin{thm} \label{thm:open}
There is an open subset of the moduli space of metric graphs of genus 4 parametrizing graphs $\Gamma$ without separating edges such that $W^1_3(\Gamma)$ has positive dimension.
\end{thm}

\noindent On this open set, the dimension of $W^1_3$ is strictly larger than the the Brill-Noether number $\rho(4,1,3)$, which is zero.

The graphs used in the proof of Theorem~\ref{thm:open} are ``loops of loops" of genus 4.  More generally, we consider \textbf{loops of loops} of genus $g \geq 3$, which are trivalent metric graphs that have $2g-2$ vertices labeled $v_1, \ldots, v_{g-1}, w_1, \ldots, w_{g-1}$, with a single edge joining $v_i$ to $w_i$, two edges joining $w_i$ to $v_{i+1}$, and two edges joining $w_{g-1}$ to $v_1$, as shown.  

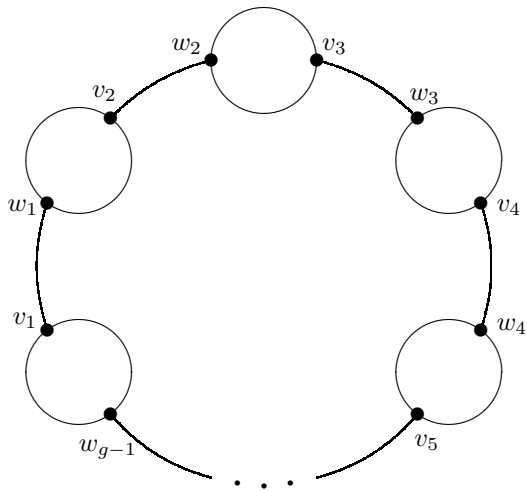
\begin{center}
\begin{picture}(200,200)
\put(30,60){\circle{40}}
\put(30,140){\circle{40}}
\put(100,178){\circle{40}}
\put(170,140){\circle{40}}
\put(170,60){\circle{40}}

\qbezier(18,76)(10,100)(18,124)
\qbezier(182,76)(190,100)(182,124)

\qbezier(42,156)(57,172)(80,178)
\qbezier(120,178)(143,172)(158,156)

\qbezier(42,44)(57,26)(80,20)
\qbezier(120,20)(143,26)(158,44)

\put(100,17){\circle*{2}}
\put(90,18){\circle*{2}}
\put(110,18){\circle*{2}}

\put(18,76 ){\circle*{5}}
\put(18,124 ){\circle*{5}}
\put(182,76 ){\circle*{5}}
\put(182,124 ){\circle*{5}}
\put(42,156 ){\circle*{5}}
\put(80,178 ){\circle*{5}}
\put(120,178 ){\circle*{5}}
\put(158,156 ){\circle*{5}}
\put(42,44 ){\circle*{5}}
\put(158,44 ){\circle*{5}}

\put(5,78){$v_1$}
\put(3,121){$w_1$}

\put(35,164){$v_2$}
\put(65,182){$w_2$}

\put(122,182){$v_3$}
\put(155,163){$w_3$}

\put(188, 121){$v_4$}
\put(188,76){$w_4$}

\put(30,30){$w_{g-1}$}
\put(157,32){$v_5$}

\end{picture}
\end{center}

\noindent Our study of ranks of divisors on such graphs depends heavily on Luo's theory of rank determining sets.  On an algebraic curve, every set of $g+1$ distinct points is rank determining, and hence every minimal rank determining set on an algebraic curve has size at most $g+1$.  Luo showed that every metric graph of genus $g$ has a rank determining set of size $g+1$, and conjectured that every minimal rank determining set should have at most this size  \cite[p.~1792]{Luo11}.

\begin{thm}  \label{thm:minimal}
Let $\Gamma$ be a loop of loops of genus $g$.  Then the set of trivalent vertices $\{v_1, w_1, \ldots, v_{g-1}, w_{g-1}\}$ is a minimal rank determining set on $\Gamma$.
\end{thm} 

\noindent This minimal rank determining set has size $2g-2$, which is greater than $g+1$ when $g$ is at least 4, giving counterexamples to Luo's conjecture.

\begin{rem}
These results for loops of loops of genus 4 are moderated by the observation that any genus 4 curve with a regular semistable model where the dual graph of the special fiber is such a loop of loops is Brill-Noether general.  This is because a genus 4 curve that is not Brill-Noether general must be hyperelliptic, and hence, by Baker's Specialization Lemma \cite{Baker08}, the dual graph of its special fiber must be a hyperelliptic graph.  Such graphs have an involution for which the quotient is a tree \cite{BakerNorine09}, and loops of loops have no such involutions.  The following definition and theorems give a framework in which one can understand and generalize this observation without appealing to special facts about hyperelliptic graphs and curves of low genus.  See Remark~\ref{rem:recover}.
\end{rem}

We propose the following definition of the Brill-Noether rank $w^r_d(\Gamma)$, as a substitute for $\dim W^r_d(\Gamma)$ when the latter is not well-behaved.

\begin{defn}  \label{def:main}
Let $\Gamma$ be a metric graph such that $W^r_d(\Gamma)$ is nonempty.  The \textbf{Brill-Noether rank} $w^r_d(\Gamma)$ is the largest integer $\rho$ such that, for every effective divisor $E$ of degree $r + \rho$, there exists a divisor $D$ of degree $d$ and rank at least $r$ on $\Gamma$ such that $D - E$ is effective.  If $W^r_d(\Gamma)$ is empty then we define $w^r_d(\Gamma)$ to be $-1$.  
\end{defn}

\noindent In many respects, the Brill-Noether rank of a graph is analogous to the dimension of the Brill-Noether locus of an algebraic curve just as the Baker-Norine rank of a divisor on a graph is analogous to the dimension of the complete linear series of a divisor on a curve; see Proposition~\ref{prop:algrank} for the classical side of this analogy.  Like Baker-Norine ranks of divisors, these Brill-Noether ranks vary upper semicontinuously in families and satisfy a specialization inequality.

\begin{thm}  \label{thm:semicont}
The function taking $\Gamma$ to $w^r_d(\Gamma)$ is upper semicontinuous on the moduli space of metric graphs.
\end{thm}

\begin{thm}  \label{thm:specialization}
Let $X$ be a smooth projective curve over a discretely valued field with a regular semistable model whose special fiber has dual graph $\Gamma$.  Then 
\[
\dim W^r_d(X) \leq w^r_d(\Gamma).
\]
\end{thm}

\noindent Since every graph with integer edge lengths is the dual graph of the special fiber of such a model, and since every metric graph is a limit of dilations of graphs with integer edge lengths, it follows that the Brill-Noether ranks of arbitrary metric graphs are essentially bounded below by the corresponding Brill-Noether numbers.

\begin{cor}
Let $\Gamma$ be a metric graph of genus $g$.  Then $w^r_d(\Gamma) \geq \min \{ \rho(g,r,d), g \}$.
\end{cor}

\noindent The paper concludes with a proof that the Brill-Noether rank takes the expected value for loops of loops in the case $(g,r,d) = (4,1,3)$.

\begin{thm}  \label{thm:w13}
Let $\Gamma$ be a loop of loops of genus 4.  Then $w^1_3(\Gamma) = 0$.
\end{thm}

\begin{rem}  \label{rem:recover}
Together, Theorems~\ref{thm:specialization} and  \ref{thm:w13} recover the fact that, if $X$ is a smooth projective curve over a discretely valued field with a regular semistable model whose special fiber has dual graph $\Gamma$, then $X$ has only finitely many divisor classes of degree 3 and rank 1.
\end{rem}

\begin{rem}
The failure of upper semicontinuity in Theorem~\ref{thm:notsemicont} also has an analogue for linear series.  The complete linear series of a divisor on a metric graph also has a natural polyhedral structure and hence a well-defined dimension \cite{HMY09}, but these dimensions do not vary upper semicontinuously in families.  See Example~\ref{ex:yu}.
\end{rem}

\begin{rem}
The specialization inequality stated here as Theorem~\ref{thm:specialization} is a close analogue of Baker's Specialization Lemma \cite[Section~2]{Baker08}, and the basic idea appeared already in \cite{tropicalBN}, where it was used to deduce the classical Brill-Noether Theorem from a ``tropical Brill-Noether Theorem".  The tropical Brill-Noether Theorem proved there gives an explicit graph $\Gamma$ for each genus $g$ with the following properties.
\begin{enumerate}
\item  If $\rho(g,r,d)$ is negative then $\Gamma$ has no divisors of degree $d$ and rank $r$.
\item  If $\rho(g,r,d)$ is zero then $\Gamma$ has exactly the expected number of distinct divisor classes of degree $d$ and rank $r$.
\item  If $\rho(g,r,d)$ is nonnegative then $\dim W^r_d(\Gamma) = \min \{\rho(g,r,d), g\}$ and $w^r_d(\Gamma) = \rho(g,r,d)$.
\end{enumerate}
These graphs are chains of $g$ loops, joined by separating vertices of valence 4.  Since the Brill-Noether theory of a graph depends only on the Jacobian and the image of the Abel-Jacobi map, and since contractions of separating edges leave these unchanged, the same results hold on trivalent chains of loops where the 4-valent separating vertices are replaced by separating edges.  Such graphs correspond to an open subset of the moduli space of metric graphs, and some have speculated that properties (1)-(3) should furthermore hold on a dense open subset of the moduli of metric graphs, just as the analogous properties hold on a dense open subset of the moduli space of curves.  Theorem~\ref{thm:open} shows that this is not the case.

Theorems~\ref{thm:open}, \ref{thm:semicont}, and \ref{thm:w13} do show that the locus where the Brill-Noether rank $w^r_d$ is equal to the Brill-Noether number $\rho(g,r,d)$ is open and strictly larger than the locus where (1)-(3) hold, but it remains unclear whether this locus is dense in the moduli space of graphs without separating edges.  See also \cite[Conjecture~6.6]{Caporaso11b} for an interesting recent conjecture on the locus of graphs with Brill-Noether general properties.
\end{rem}

\noindent \textbf{Acknowldegments.}  We thank L.~Caporaso and Y.~Luo for helpful comments on an earlier draft of this paper and J. Yu for providing Example~\ref{ex:yu}.

\section{Preliminaries}

We briefly recall the basic facts about divisors and Abel-Jacobi theory for metric graphs, following \cite{Mikhalkin06, BakerNorine07, GathmannKerber08, MikhalkinZharkov08, BakerFaber11}, to which we refer the reader for proofs, references, and further details.

\subsection{Divisors}  Let $\Gamma$ be a compact connected metric graph.  A \textbf{divisor} on $\Gamma$ is a finite $\ZZ$-linear combination of points of $\Gamma$, and we write $\Div(\Gamma)$ for the additive group of all such divisors.  The \textbf{degree} of a divisor $D = a_1 v_1 + \cdots + a_r v_r$ is the sum of its coefficients $a_1 + \cdots + a_r$.

Let $f$ be a piecewise linear function with integer slopes on $\Gamma$.  For each point $v$ in $\Gamma$, the sum of the incoming slopes of $f$ at $v$ is denoted $\ord_v(f)$.  This sum is zero for all but finitely many points of $\Gamma$, so
\[
\divisor(f) = \sum_{v \in \Gamma} \ord_v(f) \cdot v
\]
is a divisor.  The divisors of piecewise linear functions with integer slopes are called \textbf{principal} and we write $\Prin(\Gamma)$ for the subgroup of all principal divisors on $\Gamma$.  The quotient
\[
\Pic(\Gamma) = \Div(\Gamma) / \Prin(\Gamma),
\]
is called the \textbf{Picard group} of $\Gamma$ and elements of $\Pic(\Gamma)$ are called \textbf{divisor classes}.  Every principal divisor has degree zero, so the degree of a divisor class is well-defined.  We write $\Pic_d(\Gamma)$ for the subset of divisor classes of degree $d$.  In particular, $\Pic_0(\Gamma)$ is the subgroup of divisor classes of degree zero.

\subsection{Abel-Jacobi theory}  Abel-Jacobi theory for metric graphs identifies $\Pic_0(\Gamma)$ with the \textbf{Jacobian torus}
\[
\Jac(\Gamma) = \Omega^*(\Gamma) / H_1(\Gamma,\ZZ),
\]
where $\Omega^*(\Gamma)$ is the dual vector space of the real vector space of \textbf{harmonic 1-forms} on $\Gamma$, which assign a real-valued slope to each edge in $\Gamma$ in such a way that the sum of the incoming slopes is zero at every vertex.  As in the classical Abel-Jacobi theory of algebraic curves, the homology group $H_1(\Gamma, \ZZ)$ embeds as a lattice in $\Omega^*(\Gamma)$ through integration of 1-forms along 1-cycles.

For each integer $d$, the subset $\Pic_d(\Gamma)$ of divisor classes of degree $d$ is a torsor over the real torus $\Pic_0(\Gamma)$.  The \textbf{Abel-Jacobi map}
\[
\Phi : \Gamma \rightarrow \Pic_1(\Gamma)
\]
takes a point $v$ to the divisor class $[v]$.  It contracts all separating edges of $\Gamma$ and maps the resulting graph without separating edges homeomorphically onto its image.  Furthermore, this map is piecewise linear in the appropriate sense.  If we fix a basepoint then $\Pic_1(\Gamma)$ is identified with $\Pic_0(\Gamma)$ and $\Phi$ is identified with the map
\[
\Phi_w : \Gamma \rightarrow \Pic_0(\Gamma)
\]
taking a point $v$ to $[v-w]$.  Then $\Omega^*(\Gamma)$ is identified with the universal cover of $\Pic_1(\Gamma)$ and the restriction of the Abel-Jacobi map to any contractible subset $U$ factors through a piecewise-linear map to the real vector space $\Omega^*(\Gamma)$.

\subsection{Linear series} 

Let $D = a_1 v_1 + \cdots  + a_r v_r$  be a divisor on $\Gamma$.  Then $D$ is called \textbf{effective} if each of its coefficients $a_1, \ldots, a_r$ is nonnegative.  A divisor $D'$ on $\Gamma$ is \textbf{equivalent} to $D$ if $D - D'$ is a principal divisor, which exactly means that the divisor classes $[D]$ and $[D']$ are equal in $\Pic (\Gamma)$.

The \textbf{complete linear series} $|D|$ is the set of all effective divisors equivalent to $D$; it is naturally identified with the underlying set of a finite, connected polyhedral complex \cite{HMY09}.  Certain linear paths in $|D|$ are obtained by \textbf{firing subgraphs}.  The boundary of a closed subgraph $\Gamma'$ may be thought of as a divisor in which each boundary point $v_i$ appears with multiplicity equal to its out degree in $\Gamma$.  Roughly speaking, firing $\Gamma'$ for a small positive time $\epsilon$ means subtracting the boundary of $\Gamma'$ and adding the boundary of an $\epsilon$ neighborhood of $\Gamma'$.  Any two divisors in $|D|$ can be connected by a sequence of firings of subgraphs.

Since $|D|$ is the underlying set of a polyhedral complex which depends only on the class $[D]$, we may consider $\dim |D|$ as a function on $\Pic(\Gamma)$.  The complex $|D|$ is not necessarily pure dimensional, but we follow the usual convention that the dimension of a non pure complex is the maximum of the dimensions of its cells.  The following example of Josephine Yu shows that $\dim |D|$ is not upper semicontinuous on $\Pic_2$ of a genus 3 graph.

\begin{ex} \label{ex:yu}
Let $\Gamma$ be the graph of genus $3$ with four vertices labeled $v_0$, $v_1$, $w_0$, $w_1$, single edges joining $v_0$ to $v_1$ and $w_0$ to $w_1$, and pairs of edges joining $v_i$ to $w_i$, as shown.  All edges have length 1.
\smallskip
\begin{center}
\begin{picture}(200,100)
\put(60,10){\line(1,0){80}}
\put(60,10){\line(0,1){80}}
\put(60,90){\line(1,0){80}}
\put(140,10){\line(0,1){80}}
\qbezier(60,10)(25,50)(60,90)
\qbezier(140,10)(175,50)(140,90)
\put(60.5,10.5){\circle*{5}}
\put(60.5,89.5){\circle*{5}}
\put(139.5,10.5){\circle*{5}}
\put(139.5,89.5){\circle*{5}}
\put(110,10){\circle*{5}}
\put(103,17){$w_{\lambda}$}
\put(55,0){$w_0$}
\put(55,95){$v_0$}
\put(136,0){$w_1$}
\put(136,95){$v_1$}
\end{picture}
\end{center}
\smallskip
For real numbers $0 < \lambda < 1$, let $v_\lambda$ be the point in the segment $[v_0, v_1]$ at distance $\lambda$ from $v_0$.  Similarly, let $w_\lambda$ be the point in $[w_0, w_1]$ at distance $\lambda$ from $w_1$.  For $0 \leq t < 1$, we consider the divisor
\[
D_t = v_0 + w_{1-t}.
\]
When $t$ is positive, the complete linear series $R(D_t)$ is a 1-dimensional segment of length $t$, consisting of the divisors $v_\lambda + w_{1-\lambda}$ for $0 \leq \lambda \leq t$.  When $t$ is zero, the complete linear series $R(D_0)$ is the single point $v_0 + w_1$.  Since $[D_0]$ is the limit as $t$ goes to zero of the classes $[D_t]$, it follows that $\dim |D|$ is not upper semicontinuous on $\Pic_2(\Gamma)$.
\end{ex}

\subsection{Dhar's burning algorithm}
Given a fixed basepoint $v$ and an effective divisor $D$, Dhar's burning algorithm is a canonical and efficient method for finding a subgraph to fire so that the points in $D$ move toward $v$.  The algorithm terminates after finitely many steps, arriving at a divisor for which the corresponding subgraph is empty.  This divisor is called \textbf{$v$-reduced} and is characterized by the following properties.  First, the $v$-reduced divisor is effective away from $v$.  Next, among divisors equivalent to $D$ and effective away from $v$, it has the maximal possible multiplicity at $v$.   Finally, the set of distances to $v$ from the remaining points is lexicographically minimal.  Most importantly for our purposes, by the first two properties, if $D$ is $v$-reduced and does not contain $v$, then no effective divisor equivalent to $D$ contains $v$.  This burning algorithm for finding $v$-reduced divisors can be adapted to give an algorithm for determining ranks of divisors, as discussed in the following subsection.  See \cite{Dhar90, BakerNorine07, Luo11} for details.

\subsection{Ranks of divisors}

 Baker and Norine defined the rank of of an effective divisor $D$ as follows.

\begin{defn*}[\cite{BakerNorine07}]
The rank $r(D)$ is the largest integer $r$ such that, for every effective divisor $E$ of degree $r$ on $\Gamma$, there is a divisor $D'$ equivalent to $D$ such that $D' - E$ is effective.
\end{defn*}

\noindent If $D$ is not equivalent to an effective divisor then $r(D)$ is defined to be $-1$.  The rank $r(D)$ depends only on the divisor class $[D]$ in $\Pic (\Gamma)$.  The locus in $\Pic(\Gamma)$ of divisors of rank at least $r$ is closed, so the rank $r(D)$ is upper semicontinuous on $\Pic(\Gamma)$, unlike $\dim |D|$.

Luo's theory of rank determining sets shows that, in order to determine the rank of a divisor $D$, it is not necessary to test whether $D- E$ is equivalent to an effective divisor for all effective divisors $E$ of degree $r$; it suffices to check this for a finite and relatively small set of divisors.  We briefly recall the basic notions of this theory.  

For any subset $A \subset \Gamma$, the \textbf{$A$-rank} $r_A(D)$ is the largest integer $r$ such that, for every effective divisor $E$ of degree $r$ with support in $A$, there is a divisor $D'$ that is equivalent to $D$ such that $D' - E$ is effective.  Clearly $r_A(D) \geq r(D)$ for all $A$ and $D$.

\begin{defn*}[\cite{Luo11}] 
A subset $A \subset \Gamma$ is \textbf{rank determining} if $r_A(D) = r(D)$ for all divisors $D$ on $\Gamma$.
\end{defn*}

\noindent  In the same paper where he introduced this notion, Luo proved that every graph $\Gamma$ of genus $g$ has a rank determining set of size $g+1$, showed that rank determining sets are preserved by homeomorphisms, and gave necessary and sufficient topological criteria for a finite subset to be rank determining.

\section{Brill-Noether loci of metric graphs}  \label{sec:BNloci}

If $X$ is a smooth projective algebraic curve then, for nonnegative integers $r$ and $d$, the Brill-Noether locus $W^r_d(X) \subset \Pic_d(X)$ is the subset parametrizing divisor classes of degree $d$ whose complete linear series have dimension at least $r$.  This subset carries a natural scheme structure, given by its realization as a degeneracy locus of a natural map of vector bundles over $\Pic_d(X)$.  See \cite[Chapter~4]{ACGH} for details.

We define the Brill-Noether locus of the metric graph $\Gamma$ as follows.

\begin{defn}
For nonnegative integers $r$ and $d$, the \textbf{Brill-Noether locus}
\[
W^r_d(\Gamma) \subset \Pic_d(\Gamma)
\]
is the set of divisor classes of degree $d$ and rank at least $r$.
\end{defn}

This Brill-Noether locus carries a natural topology as a subspace of the torus torsor $\Pic_d(\Gamma)$.  We will show that $W^r_d(\Gamma)$ is, roughly speaking, the underlying set of a closed polyhedral complex.  To make this precise, we define polyhedral subsets of $\Pic_d(\Gamma)$ as follows.

Fix a basepoint $w$ in $\Gamma$.  Then the map taking a divisor class $[D]$ of degree $d$ to $[D - dw]$ identifies $\Pic_d(\Gamma)$ with $\Pic_0(\Gamma)$.  In particular, the choice of basepoint identifies $\Omega^*(\Gamma)$ with the universal cover of $\Pic_d(\Gamma)$.

\begin{defn}
A subset of $\Pic_d(\Gamma)$ is \textbf{polyhedral} if it is the image of a finite union of polytopes in $\Omega^*(\Gamma)$.
\end{defn}

\noindent  A different choice of basepoint will change the map from $\Omega^*(\Gamma)$ to $\Pic_d(\Gamma)$ by a translation, so the notion of polyhedral subsets of $\Pic_d(\Gamma)$ is well-defined, independent of the choice of basepoint.  Polyhedral subsets are always closed, and the union of any finite number of polyhedral subsets of $\Pic_d(\Gamma)$ is polyhedral.

\begin{ex}  \label{ex:AJ}
The restriction of the Abel-Jacobi map to each edge of $\Gamma$ factors through a linear map to $\Omega^*(\Gamma)$.  Therefore, the image $\Phi(\Gamma)$ is a polyhedral subset of $\Pic_1(\Gamma)$.
\end{ex}

\begin{lem}  \label{lem:int}
The intersection of any finite number of polyhedral subsets of $\Pic_d(\Gamma)$ is polyhedral.
\end{lem}

\begin{proof}
If we fix a basis for $H_1(\Gamma, \ZZ)$ then $\Pic_d(\Gamma)$ is obtained from the unit cube in $\Omega^*(\Gamma)$ with respect to this basis by identifying opposite faces, and a subset of $\Pic_d(\Gamma)$ is polyhedral if and only if its preimage in the unit cube is a finite union of polytopes.  The lemma follows, since any finite intersection of finite unions of polytopes is again a finite union of polytopes.
\end{proof}

\begin{lem}  \label{lem:sums}
If $S$ and $S'$ are polyhedral subsets of $\Pic_d(\Gamma)$ and $\Pic_{d'}(\Gamma)$ then the sumset $S + S'$ is a polyhedral subset of $\Pic_{d + d'}(\Gamma)$.
\end{lem}

\begin{proof}
Say $S$ and $S'$ are the images of the finite unions of polytopes $P_1 \cup \cdots \cup P_k$ and $P'_1 \cup \cdots \cup P'_\ell$ in $\Omega^*(\Gamma)$, respectively.  Then $S + S'$ is the union of the images of the Minkowski sums $P_i + P'_j$.
\end{proof}

For a nonnegative integer $d$, let $\Eff_d(\Gamma) \subset \Pic_d(\Gamma)$ be the set of classes of effective divisors of degree $d$ on $\Gamma$.

\begin{prop} \label{prop:Eff}
The set of effective classes $\Eff_d (\Gamma)$ is a polyhedral subset of dimension $\min\{d,g\}$ in $\Pic_d(\Gamma)$.
\end{prop}

\begin{proof}
The fact that $\Eff_d$ is polyhedral of dimension at most $d$ follows from Example~\ref{ex:AJ} and Lemma~\ref{lem:sums}, since $\Eff_d(\Gamma)$ is the sum of $d$ copies of the 1-dimensional polyhedral image of the Abel-Jacobi map.  Since dimensions of sumsets are subadditive and $\Eff_g(\Gamma)$ is the full torus $\Pic_g(\Gamma)$, the dimension of $\Eff_d(\Gamma)$ must be equal to $d$ for $1 \leq d \leq g$.  
\end{proof}

\begin{prop}  \label{prop:polyhedral}
The Brill-Noether locus $W^r_d(\Gamma)$ is a polyhedral subset of $\Pic_d(\Gamma)$.
\end{prop}

\begin{proof}
Fix a finite rank determining set $A$ for $\Gamma$.  Then there are finitely many effective divisors $E_1, \ldots, E_k$ of degree $r$ with support contained in $A$, and $W^r_d(\Gamma)$ is the intersection of the images of the maps
\[
\varphi_i : \Eff_{d-r}(\Gamma) \rightarrow \Pic_d(\Gamma)
\]
taking the class of an effective divisor $D$ to $[D + E_i]$.  The proposition then follows from Lemma~\ref{lem:int} and Proposition~\ref{prop:Eff}; the image of each $\varphi_i$ is polyhedral, since it is a translation of $\Eff_{d-r}(\Gamma)$.  Then $W^r_d(\Gamma)$ is the intersection of these finitely many polyhedral subsets, and hence polyhedral by Lemma~\ref{lem:int}.
\end{proof}

\begin{rem}
Like the complete linear series of a divisor on a metric graph, the Brill-Noether locus $W^r_d(\Gamma)$ is polyhedral but not necessarily pure dimensional, so $\dim W^r_d(\Gamma)$ refers to the maximum of the dimensions of its polyhedral cells.  It is also worth noting that $|D|$ is a contractible complex in a vector space, while $W^r_d(\Gamma)$ often has nontrivial topology and lives in the torus $\Pic_d(\Gamma)$.
\end{rem}

%\begin{prop}
%Let $X$ be a smooth projective curve of genus $g$ over a discretely valued field with a regular semistable model whose special fiber has dual graph $\Gamma$.  Then
%\[
%\dim W^r_d(X) \leq \dim W^r_d(\Gamma).
%\]
%\end{prop}

%\begin{proof}
%TO BE WRITTEN -- Use SGA 7 and Gubler's paper.
%\end{proof}

%There are graphs for which this specialization inequality is always strict.  In fact, there are open subsets of such graphs in the moduli space of metric graphs.

\section{Loops of loops}

Let $\Gamma$ be a loop of loops of genus $g$, as described in the introduction.   So $\Gamma$ is a trivalent metric graph whose vertices are labeled $v_1, \ldots, v_{g-1}, w_1, \ldots w_{g-1}$ with a single edge of length $\ell_i$ joining $v_i$ to $w_i$, two edges joining $w_i$ to $v_{i+1}$ for $1 \leq i \leq g-2$ and two edges joining $w_{g-1}$ to $v_1$.  The case $g = 4$ is pictured here.

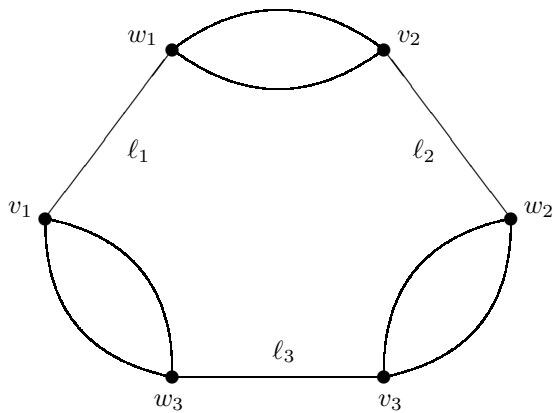
\begin{center}
\begin{picture}(200,165)
\put(60,20){\line(1,0){80}}
\put(12,80){\line(3,4){48}}
\put(188,80){\line(-3,4){48}}
\qbezier(60,144)(100,174)(140,144)
\qbezier(60,144)(100,114)(140,144)
\qbezier(12,80)(11,30)(60,20)
\qbezier(12,80)(61,70)(60,20)
\qbezier(188,80)(189,30)(140,20)
\qbezier(188,80)(139,70)(140,20)
\put(60,20){\circle*{5}}
\put(140,20){\circle*{5}}
\put(12,80){\circle*{5}}
\put(60,144){\circle*{5}}
\put(188,80){\circle*{5}}
\put(140,144){\circle*{5}}
\put(-2,82){$v_1$}
\put(193,82){$w_2$}
\put(43,147){$w_1$}
\put(145,147){$v_2$}
\put(53,9){$w_3$}
\put(138,9){$v_3$}
\put(98,27){$\ell_3$}
\put(43,103){$\ell_1$}
\put(151,103){$\ell_2$}
\end{picture}
\end{center}

We study ranks of divisors on $\Gamma$ using Luo's theory of rank determining sets.  This theory builds on Dhar's burning algorithm \cite{Dhar90} and the properties of $v$-reduced divisors developed in \cite[Section~3]{BakerNorine07}.  See also \cite[Section~2]{Luo11} for details on these basic notions, which we will use freely, without further mention.  

Recall that Luo defines an open subset of a metric graph to be \textbf{special} if it is connected and every connected component of its complement has a vertex with out degree at least two.  It follows that the closure of a special open set is a connected subgraph of positive genus. A subset $A \subset \Gamma$ is rank determining if there are no special open sets in the complement of $A$; see Definition~3.2 and Theorem~3.8 in \cite{Luo11}.

\begin{lem}
The set $A = \{ v_1, \ldots, v_{g-1}, w_1, \ldots, w_{g-1} \}$ is rank determining in $\Gamma$.
\end{lem}

\begin{proof}
The closure of any connected component of $\Gamma \smallsetminus A$ is a tree, so the closure of any connected open subset of $\Gamma \smallsetminus A$ has genus zero.  It follows that there are no special open sets in the complement of $A$.
\end{proof}

To prove Theorem~\ref{thm:minimal} we must show that no proper subset of $A$ is rank determining.  The following example gives an explicit divisor on a loop of loops of genus 4 whose rank is not determined by $A \smallsetminus \{w_3\}$.

\begin{ex}
Suppose $g = 4$, the lengths $\ell_1$ and $\ell_2$ are both 1, and $\ell_3$ is 2.  Let $D = v_1 + w_2 + v_3$.  Firing the genus 2 subgraph bounded by $v_1$ and $w_2$ shows that $D$ is equivalent to $w_1 + v_2 + v_3$.  Since $D$ is equivalent to an effective divisor containing any point of $A' = A \smallsetminus \{w_3\}$, we have
\[
r_{A'}(D) \geq 1.
\]
However, $D$ is also equivalent to $D' = v_1 + v_2 + w$, where $w$ is the midpoint of the edge $[v_3, w_3]$, and it is straightforward to check by Dhar's burning algorithm that $D'$ is $w_3$-reduced.  Since $D'$ does not contain $w_3$, it follows that no effective divisor equivalent to $D$ contains $w_3$.  Therefore $r(D)= 0$, and $A'$ is not rank determining.
\end{ex}

Since rank determining sets are preserved by homeomorphisms \cite[Theorem~1.10]{Luo11}, and homeomorphisms of $\Gamma$ act transitively on $A$, it follows that no proper subset of $A$ is rank determining for $g = 4$.  We now prove the general case.

\begin{proof}[Proof of Theorem~\ref{thm:minimal}]
To show that the rank determining set $A$ is minimal, it will suffice to show that, for each point $v$ of $A$ there is a special open set in $\Gamma$ whose intersection with $A$ is exactly $\{v\}$.  See \cite[Proposition~3.26]{Luo11}.  Since homeomorphisms of $\Gamma$ act transitively on $A$, it will suffice to consider $v = v_1$.  Let $U$ be the connected open neighborhood of $v_1$ bounded by $w_1$ and $w_{g-1}$.  The complement of $U$ is connected and $w_1$ has outdegree 2, so $U$ is special.  Since the intersection of $U$ with $A$ is exactly $\{v\}$, the theorem follows.
\end{proof}

We now return to the case $g = 4$ and show that there is an open set of loops of loops $\Gamma$ such that the Brill-Noether locus $W^1_3(\Gamma)$ has positive dimension.

\begin{proof}[Proof of Theorem~\ref{thm:open}]
Let $\Gamma$ be a loop of loops of genus 4.  Suppose $\ell_1 > \ell_2 > \ell_3$ and $\ell_2 + \ell_3 > \ell_1$.  The set of all such graphs is open in the moduli space of metric graphs without separating edges.  Therefore, to prove the theorem it will suffice to show that $\dim W^1_3(\Gamma) \geq 1$.

Let $D$ be a divisor of the form $v_1 + w_3 + w$, where $w$ is in the interval $[v_2, w_2]$ at distance at least $\ell_1 - \ell_3$ from $v_2$.  We claim that $D$ has rank at least 1.  The theorem follows from this claim, the set of classes of all such divisors in $\Pic_3(\Gamma)$ is an embedded interval.  

It remains to show that $D$ has rank at least 1.  Firing the genus 2 subgraph of $\Gamma$ bounded by $v_1$ and $w$ shows that $D$ is equivalent to $v' + v_2 + w_3$ with $v'$ in the interval $[v_1, w_1]$ at distance $\ell_1 - d(v_2, w)$ from $v_1$.  Similarly, firing the loop bounded by $v_1$ and $w_3$ shows that $D$ is equivalent to
\[
D' = v + w + v_3,
\]
where $v$ is the point in $[v_1, w_1]$ at distance $\ell_3$ from $v_1$.  

Now, starting from $D'$ and firing the genus 1 subgraph bounded by $v$ and $w$ shows that $D$ is equivalent to $v_3 + w_2 + v'$, with $v'$ in the segment $[v_1, w_1]$.  Similarly, starting from $D'$ and firing the genus 2 subgraph bounded by $v$ and $w$ shows that $D$ is equivalent to $w_1 + w' + v_3$, with $w'$ in the segment $[v_2, w_2]$.  Altogether, this shows that $D$ is linearly equivalent to effective divisors that contain each element of $A$.  Therefore $r_A(D) = 1$ and, since $A$ is rank determining, $r(D) = 1$, as required.
\end{proof}

We conclude this section by using limits of loops of loops of genus 4 to show that $\dim W^r_d$ is not upper semicontinuous on the moduli space of bridgeless metric graphs.

\begin{proof}[Proof of Theorem~\ref{thm:notsemicont}]
Let $\Gamma_0$ be the degenerate loop of loops with three vertices $v_1$, $v_2$, and $v_3$, where each pair of distinct vertices is joined by a pair of edges of length 1.
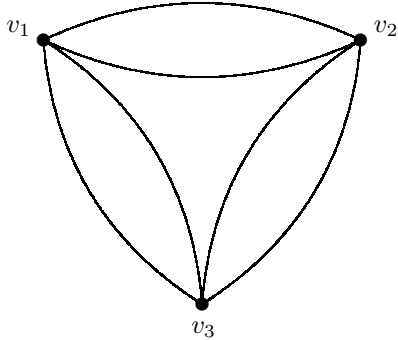
\begin{center}
\begin{picture}(200,120)(0,10)
\put(40,120){\circle*{5}}
\put(160,120){\circle*{5}}
\put(100,20){\circle*{5}}
\qbezier(40,120)(45,55)(100,20)
\qbezier(40,120)(95,85)(100,20)
\qbezier(100,20)(155,55)(160,120)
\qbezier(100,20)(105,85)(160,120)
\qbezier(40,120)(100,148)(160,120)
\qbezier(40,120)(100,92)(160,120)
\put(26,123){$v_1$}
\put(165,123){$v_2$}
\put(96,9){$v_3$}
\end{picture}
\end{center}

Fix positive real numbers $\ell_1$, $\ell_2$, and $\ell_3$ such that $\ell_1 > \ell_2 > \ell_3$ and $\ell_2 + \ell_3 > \ell_1$, as in the proof of Theorem~\ref{thm:open}.  Let $\Gamma_t$ be the loop of loops of genus 4 in which $[v_i, w_i]$ has length $t \cdot \ell_i$ and all other edges have length 1.  Then $\Gamma_0$ is the limit of $\Gamma_t$ as $t$ goes to zero.  In the proof of Theorem~\ref{thm:open} we showed that $W^1_3(\Gamma_t)$ has positive dimension for $t > 0$. 

We claim that $W^1_3(\Gamma_0)$ consists of the single rank 1 class $[v_1 + v_2 + v_3]$ and hence is zero dimensional.  Indeed, let $D$ be a divisor of degree 3 and rank 1 on $\Gamma$.  Replacing $D$ by an equivalent divisor, we may assume $D$ is $v_1$-reduced, so $D = v_1 + v + w$ for some points $v$ and $w$ in $\Gamma$.  Dhar's burning algorithm shows that, since $D$ is $v_1$-reduced, the points $v$ and $w$ cannot both be in the interior of the same edge.  Applying Dhar's algorithm again, from $v_2$ and $v_3$, shows that $D$ is $v_2$-reduced and $v_3$-reduced, as well.  Since $r(D) = 1$, by hypothesis, and the set $\{v_1, v_2, v_3\}$ is rank determining,  it follows that $D$ must contain $v_2$ and $v_3$.  Therefore, $[D] = [v_1 + v_2 + v_3]$, as claimed.

We have shown that $\dim W^1_3(\Gamma_t)$ is positive for all positive $t$ and $\dim W^1_d(\Gamma_0)$ is zero.  Therefore, since $\Gamma_0$ is the limit of $\Gamma_t$ as $t$ goes to zero, $\dim W^1_3$ is not upper semicontinuous on the moduli space of metric graphs.
\end{proof}

\section{Brill-Noether rank}

In this final section, we show that the ranks of Brill-Noether loci of metric graphs, as defined in the introduction, vary upper semicontinuously in families and are related to dimensions of algebraic Brill-Noether loci by a specialization inequality.

Let $G$ be a connected graph.  Label the vertices of $G$ by $v_1, \ldots, v_m$ and the edges by $e_1, \ldots, e_n$.  So the genus of a geometric realization is $g = n-m + 1$.  We consider
\[
\sigma = \RR_{\geq  0}^n
\]
as a parameter space for possibly degenerate metric realizations of $G$; the point $\ell = (\ell_1, \ldots, \ell_s)$ corresponds to the metric graph $\Gamma_\ell$ in which $e_i$ has length $\ell_i$.  If $\ell_i = 0$, this produces a degenerate realization in which $e_i$ is contracted.  

Over $\sigma$, there is a universal family $\Gamma_\sigma$ of possibly degenerate metric realizations of $G$, obtained by gluing the cones
\[
\gamma_i = \{(\ell_1, \ldots, \ell_n, t) \in \RR_{\geq 0}^{n + 1} \ | \ 0 \leq t \leq \ell_i \}.
\]
The gluing depends on the choice of an orientation for each edge, but the resulting metric space is independent of all choices, as is the natural projection to $\sigma$ obtained by forgetting the last coordinate.  The fiber over $\ell$ is the metric graph $\Gamma_\ell$, and the intersection of this fiber with $\gamma_i$ is the edge $e_i$ of length $\ell_i$.

We now describe the subspace of $\sigma$ that parametrizes possibly degenerate realizations of $G$ that have genus $g$.  For each subset $I \subset \{ 1, \ldots, n \}$, let $G_I$ be the subgraph whose geometric realization is the union of the edges $e_i$ for $i \in I$, and let $g_I$ be the first Betti number of $G_I$.  Let $\tau_I$ be the face of $\sigma$ where $\ell_i = 0$ for $i \in I$.  If $\ell$ is in the relative interior of $\tau_I$ then $\Gamma_\ell$ has genus $g - g_I$.  In particular, the open subset
\[
\sigma^* = \sigma \smallsetminus \bigcup_{g_I > 0} \tau_I
\]
parametrizes possibly degenerate metric realizations of $G$ with genus $g$.  The moduli space of metric graphs of genus $g$ is the colimit of a natural diagram of such open cones for all combinatorial graphs of genus $g$.  Therefore, to show the Brill-Noether rank is upper semicontinuous it will suffice to prove this on $\sigma^*$.

We write $\Gamma^*$ for the preimage of $\sigma^*$ in $\Gamma_\sigma$.    The natural piecewise linear parame\-tri\-zation map $\sigma^* \times \Gamma_{(1, \ldots, 1)} \rightarrow \Gamma^*$ in which the edge $e_i$ is stretched uniformly by a factor of $\ell_i$ in the fiber over $\ell$ is a homotopy equivalence, as is the inclusion of any fiber $\Gamma_\ell \subset \Gamma^*$ for $\ell \in \sigma^*$.  The dual of the space of harmonic forms $\Omega^*(\Gamma_\ell)$ is naturally identified with $H_1(\Gamma_\ell, \RR)$ for each $\ell \in \sigma^*$, and hence with $H_1(\Gamma^*, \RR)$.  We fix the vertex $v_1$ as a basepoint and define the relative Jacobian as
\[
 \Pic_0 (\Gamma^*) = \sigma^*  \times \big( H_1(\Gamma^*, \RR) / H_1(\Gamma^*, \ZZ) \big) .
 \]
 The standard arguments in Abel-Jacobi theory for a single graph then produce a piecewise linear relative Abel-Jacobi map over $\sigma^*$
\[
\Phi: \Gamma^* \rightarrow  \Pic_0(\Gamma^*),
\]
compatible with projections to $\sigma^*$, whose base change to $\ell \in \sigma^*$ is the usual Abel-Jacobi map for $\Gamma_\ell$.  

We claim that the universal family of realizations $\pi:\Gamma^* \rightarrow \sigma^*$ is proper, in the topological sense, meaning that the preimage of any compact set in $\sigma^*$ is compact in $\Gamma^*$.  To see this, just note that if $C$ is any subset of $\sigma^*$ then $\pi^{-1}(C)$ is the continuous image of $C \times \Gamma_{(1,\ldots, 1)}$ under the natural piecewise linear parametrization map described above.  Therefore, if $C$ is compact then $\pi^{-1}(C)$ is a continuous image of the compact space $C \times \Gamma_{(1, \ldots, 1)}$, and hence is also compact.  Since proper maps are universally closed and $\Phi$ commutes with projection to $\sigma^*$, it follows that the image of the universal graph $\Phi(\Gamma^*)$ is closed in the universal Jacobian $\Pic_0(\Gamma^*)$.  We consider $\Pic_0(\Gamma^*)$ as a group object in the category of topological spaces over $\sigma^*$, so the set $\Eff_d(\Gamma^*)$ parametrizing graphs $\Gamma_\ell$ with the class of an effective divisor of degree $d$ is the sumset of $d$ copies of $\Phi(\Gamma^*)$, and hence is also closed.

We now show that the function taking a graph $\Gamma$ to $w^r_d(\Gamma)$ is upper semicontinuous.

\begin{proof}[Proof of Theorem~\ref{thm:semicont}]
With the fixed basepoint $v_1$, the torus torsor $\Pic_d(\Gamma_\ell)$ is identified with $\Pic_0(\Gamma_\ell)$ for all $\ell$, so we consider Brill-Noether loci inside the torus $\Pic_0$ instead of inside the torsor $\Pic_d$.  

First, we claim that the universal Brill-Noether locus $W^r_d(\Gamma^*) = \bigsqcup_\ell W^r_d(\Gamma_\ell)$ is closed in $\Pic_0(\Gamma^*)$.  To see this, let $v_i^*$ be the section of $\Gamma^*$ given by the vertex $v_i$ for $1 \leq i \leq m$.  Then, for any tuple of nonnegative integers $a = (a_1, \ldots, a_m)$ such that $a_1 + \cdots + a_m = r$, we have the closed subset
\[
S_a = a_1 v_1^* + \cdots + a_m v_m^* + \Eff_{d-r}(\Gamma^*)
\]
in $\Eff_d(\Gamma^*)$.  Since the vertex set $\{v_1, \ldots, v_m\}$ is rank determining on $\Gamma_\ell$ for all $\ell$, by \cite{Luo11}, the universal Brill-Noether locus over $\sigma^*$ is
\[
W^r_d(\Gamma^*) = \bigcap_a S_a,
\]
which is an intersection of closed sets and hence closed.

Now, consider the map
\[
\mu: \Eff_{d-r-\rho}(\Gamma^*) \times \Eff_{r+\rho}(\Gamma^*) \rightarrow \Pic_d(\Gamma^*),
\]
given by adding effective divisor classes.  We consider the closed set $\mu^{-1}(W^r_d(\Gamma^*))$ and its image $Z \subset \Eff_{r+\rho}(\Gamma^*)$ under projection to the second factor.  This closed set $Z$ parametrizes graphs $\Gamma_\ell$ with the class of an effective divisor of degree $(r + \rho)$ that is contained in an effective divisor of degree $d$ and rank at least $r$.  Now, by definition, $w^r_d(\Gamma_\ell)$ is at least $\rho$ if and only if $Z$ contains $\Eff_{r+ \rho}(\Gamma_\ell)$.  Therefore, we consider the complement
\[
U = \Eff_{r + \rho}(\Gamma^*) \smallsetminus Z,
\]
which is open in $\Eff_{r+\rho}(\Gamma^*)$ and parametrizes graphs $\Gamma_\ell$ with the class of an effective divisor of degree $(r + \rho)$ that is not contained in any effective divisor of degree $d$ and rank $r$.  It remains to show that the image of $U$ is open in $\sigma^*$.

We claim that the projection from $\Eff_{r + \rho}(\Gamma^*)$ to $\sigma^*$ is open.  Indeed, the projection $p_2$ from $(\Gamma_{(1, \ldots, 1)})^{r + \rho} \times \sigma^*$ to $\sigma^*$ factors through the natural parametrizing map onto $\Eff_{r + \rho}(\Gamma^*)$.  Since $p_2$ is an open mapping, it follows that the image of $U$ is open, which proves the claim. We have shown that the set of $\ell$ in $\sigma^*$ such that $w^r_d(\Gamma_\ell)$ is less than $\rho$ is open, for arbitrary $\rho$. Therefore $w^r_d$ is upper semicontinuous, as required.
\end{proof}

We now explore the relationship between Brill-Noether ranks of metric graphs and dimensions of Brill-Noether loci of algebraic curves.  Consider a smooth projective algebraic curve $X$ of genus $g$, and suppose the Brill-Noether locus $W^r_d(X)$ is nonempty.  Let $D$ be an effective divisor on $X$ whose class is in $W^r_d(X)$.  

Roughly speaking, this class being in $W^r_d(X)$ means that $D$ is a configuration of $d$ points on $X$ that can be moved in a family parametrized by the projective line to contain any configuration of $r$ points on $X$.  Now, suppose $[D]$ is contained in a positive dimensional component of $W^r_d(X)$ and $E$ is an effective divisor of degree $r$ on $X$.  Then $D$ is equivalent to some divisor $D'$ such that $D' - E$ is effective.  Moving $[D]$ in a family $[D_t]$ parametrized by a (nonrational) algebraic curve in $W^r_d(X)$ should lead to a family $D'_t$ of divisors of degree $d$ that contain $E$ and whose classes lie in $W^r_d(X)$, and the residual divisors $D'_t - E$ should be configurations of $d - r$ points that move in a family sweeping out all of $X$.  Since $E$ is arbitrary, this would mean that any effective divisor of degree $r + 1$ on $X$ is contained in an effective divisor whose complete linear series has dimension at least $r$.  The following proposition makes this rough idea precise, and extends it to the case where the Brill-Noether locus has dimension greater than 1.

\begin{prop} \label{prop:algrank}
Let $X$ be a smooth projective curve.  Suppose $W^r_d(X)$ is not empty, and let $E$ be an effective divisor of degree $r + \dim W^r_d(X)$ on $X$.  Then there is a divisor $D$ whose class is in $W^r_d(X)$ such that $D-E$ is effective.
\end{prop}

\begin{proof}
Consider the subset $S \subset X^d$ consisting of tuples $(x_1, \ldots, x_d)$ such that the divisor class $[x_1 + \cdots + x_d]$ is in $W^r_d(X)$.  In other words, $S$ is the preimage of $W^r_d(X)$ under the natural map $\phi: X^d \rightarrow \Pic_d(X)$.  The fiber of this map over a divisor class $[D]$ is invariant under the action of the symmetric group on the $d$ factors, and the quotient is the complete linear series $|D|$.  If $[D]$ is in $W^r_d(X)$, then the fiber $\phi^{-1}([D])$ surjects onto $X^r$ under the projection to the first $r$ factors.  Therefore, $S$ has dimension at least $r + \dim W^r_d(X)$.

Choose $k$ as large as possible such that projection to the first $k$ factors maps $S$ surjectively onto $X^k$.  We must show that $k$ is at least $r + \dim W^r_d(X)$.

Suppose $k$ is less than $r + \dim W^r_d(X)$, so the general fiber of the projection $\pi: S \rightarrow X^k$ is positive dimensional.  Since $S$ is proper, the fiber dimension is upper semicontinuous, and hence every fiber is positive dimensional.  Let $x = (x_1, \ldots, x_k)$ be a point in $X^k$, and consider the projection of $\pi^{-1}(x)$ onto the $i$th coordinate of $X^d$, for $i > k$.  Each of these projections has the same image, by symmetry.  Therefore, the image must be 1-dimensional, and hence $\pi^{-1}(x)$ surjects on $X$.  Since this holds for all $x$ in $X^k$, it follows that $S$ surjects onto $X^{k+1}$, contradicting the choice of $k$.  We conclude that $k$ is at least $r + \dim W^r_d(X)$, and the proposition follows.
\end{proof}

\noindent The proposition gives further justification for the rough idea that the rank of the Brill-Noether locus of a metric graph is an avatar for the dimension of the Brill-Noether locus of an algebraic curve.  We now use the proposition to prove the specialization inequality stated in the introduction, which says that the Brill-Noether rank can only go up when specializing from curves to graphs.

\begin{proof}[Proof of Theorem~\ref{thm:specialization}]
Let $X$ be a smooth projective curve of genus $g$ over a discretely valued field with a regular semistable model whose special fiber has dual graph $\Gamma$, and suppose $W^r_d(X)$ is nonempty.  We must show that any effective divisor $E$ of degree $r + \dim W^r_d(X)$ on $\Gamma$ is contained in an effective divisor of degree $d$ and rank $r$.  First, we prove this in the case where $E$ is rational, i.e. supported at points a rational distance from the vertices of $\Gamma$, and then we prove the general case by rational approximation.

Let $E$ be an effective divisor of degree $r + \dim W^r_d(X)$ on $\Gamma$ that is rational.  Then $E$ is the specialization of an effective divisor $E'$ on $X$.  By Proposition~\ref{prop:algrank} there is an effective divisor $D'$ of degree $d$ and rank at least $r$ on $X$ that contains $E'$.  Then the specialization of $D'$ is an effective divisor of
degree $d$ and rank at least $r$ on $\Gamma$ that contains $E$.

Now, let $E$ be an arbitary, possibly nonrational, effective divisor of
degree $r + \dim W^r_d(X)$ on $\Gamma$.  Choose a sequence $\{E_1, E_2, \ldots \}$ of rational
divisors on $\Gamma$ that converges to $E$.  For each $E_i$ let $D_i$ be an
effective divisor of degree $d$ and rank at least $r$ that contains $E_i$.
Since the $d$th symmetric product of $\Gamma$ is compact, some subsequence
of $\{D_1, D_2, \ldots \}$ converges to a divisor $D$.  Now, since $W^r_d(\Gamma)$
is closed and contains $[D_i]$ for all $i$, the limit $D$ must have rank at
least $r$.  Then $D$ is an effective divisor of
degree $d$ and rank at least $r$ on $\Gamma$ that contains $E$, and the theorem follows.
\end{proof}

We conclude by showing that the Brill-Noether $w^1_3$ takes the expected value $\rho(4,1,3) = 0$ for a loop of loops of genus 4.

\begin{proof}[Proof of Theorem~\ref{thm:w13}]
Let $\Gamma$ be a loop of loops of genus 4.  We may assume that $[v_1, w_1]$ is the longest of the three single edges, i.e. that $\ell_1$ is greater than or equal to $\ell_2$ and $\ell_3$.  We claim that $v_1 + w_1$ is not contained in any effective divisor of degree 3 and rank at least 1.

Let $v_1 + w_1 + w$ be an effective divisor of degree 3 that contains $v_1 + w_1$.  The following case by case analysis shows that $D$ has rank zero by exhibiting, in each case, a vertex $v$ of $\Gamma$ such that the $v$-reduced divisor equivalent to $D$ does not contain $v$.

\bigskip

\noindent \emph{Case 1:}  Suppose $w$ is contained in either $[v_1, w_1]$ or one of the open segments $(w_1, v_2)$, $(w_2, v_3)$, or $(w_3, v_1)$.  Then $D$ is $v_2$-reduced, but does not contain $v_2$.

\bigskip

\noindent \emph{Case 2:}  Suppose $w$ is contained in the segment $[v_2, w_2]$.  Firing the genus 1 subgraph bounded by $v_1$ and $w$ shows that $D$ is equivalent to $D' = v_1 + w' + w_2$, where $v'$ is in the segment $[v_1, w_1]$.  Then $D'$ is $v_3$-reduced but does not contain $v_3$.

\bigskip

\noindent \emph{Case 3:}  Suppose $w$ is contained in the segment $[v_3, w_3]$.  Firing the genus 1 subgraph bounded by $v_1$ and $w$ shows that $D$ is equivalent to $D'' = v' + w_1 + v_3$, where $v'$ is in the segment $[v_1, w_1]$.  Then $D''$ is $v_2$-reduced but does not contain $v_2$.
\end{proof}

\bibliographystyle{amsalpha}
\providecommand{\bysame}{\leavevmode\hbox to3em{\hrulefill}\thinspace}
\providecommand{\MR}{\relax\ifhmode\unskip\space\fi MR }
\providecommand{\MRhref}[2]{%
  \href{http://www.ams.org/mathscinet-getitem?mr=#1}{#2}
}
\providecommand{\href}[2]{#2}

\end{document}